%% file: main.tex
\documentclass[11pt]{article}
\usepackage[a4paper,margin=27mm]{geometry}
\usepackage[T1]{fontenc}
\usepackage{lmodern,microtype,amsmath,amssymb,amsthm,booktabs,longtable,graphicx,xurl}
\usepackage{xcolor}
\usepackage[colorlinks=true,linkcolor=blue!45!black,citecolor=blue!45!black,urlcolor=blue!45!black]{hyperref}
\hypersetup{pdftitle={Bistability with equal and independent removal in small bimolecular networks},pdfauthor={Bo Peng}}
\newtheorem{theorem}{Theorem}
\newtheorem{lemma}[theorem]{Lemma}
\newcommand{\Rp}{\mathbb R_{>0}}
\newcommand{\diag}{\operatorname{diag}}
\newcommand{\tr}{\operatorname{tr}}
\title{Bistability with equal and independent removal\\in small bimolecular networks}
\author{Bo Peng\thanks{\raggedright Institute of Mathematical Sciences, ShanghaiTech University.
Correspondence: \href{mailto:pengbo2025@shanghaitech.edu.cn}{pengbo2025@shanghaitech.edu.cn}.}}
\date{23 October 2026}
\begin{document}
\maketitle
\begin{abstract}
We classify two-species mass-action networks with at most three internal
directed reactions on $\{X,Y,2X,X+Y,2Y\}$, positive inflows, and linear removal
of both species. Among the 681 networks up to species exchange, exactly 27
admit two distinct positive hyperbolic attracting equilibria. The same 27
admit bistability with equal removal rates and with independently chosen
positive removal rates. Each is minimal under reaction deletion. We give
rational equal-removal witnesses for all positive cases and all-parameter
exclusions for the remaining 654. The proof combines nullcline geometry,
polynomial sign tests, resultant root bounds, and exact flux-cone
decompositions. A separate rational-arithmetic implementation verifies the
complete classification and its equilibrium certificates.
\end{abstract}

\section{Models and classification}
We compare the capacity for bistability under common dilution and
species-specific removal in two-species bimolecular networks with at most
three internal reactions. Parameters are chosen freely within each model.

Let $\mathcal C=\{(1,0),(0,1),(2,0),(1,1),(0,2)\}$. An internal network $R$
contains at most three distinct directed edges $\alpha_j\to\beta_j$ on
$\mathcal C$, with $\alpha_j\ne\beta_j$. Each edge has rate $k_j>0$;
a reversible pair counts as two edges. Write
\[
 H_R(z;k)=\sum_j k_j(\beta_j-\alpha_j)x^{\alpha_{j1}}y^{\alpha_{j2}},
 \qquad z=(x,y)^T.
\]
With both species supplied and removed, the normalized models are
\begin{equation}\label{eq:models}
 \mathrm E:\quad\dot z=(a,b)^T-z+H_R(z;k),\qquad
 \mathrm F:\quad\dot z=(a,b)^T-\diag(\rho,1)z+H_R(z;k),
\end{equation}
where $a,b,k_j,\rho>0$. Thus four species-flow reactions are compulsory;
all other edges have nonzero source and target in $\mathcal C$.
The target is two distinct points in $\Rp^2$ satisfying
$F=0$, $\tr DF<0$, and $\det DF>0$. Let $\mathcal B_E$ and $\mathcal B_F$
denote the corresponding sets of topologies. Then
$\mathcal B_E\subseteq\mathcal B_F$.

There are $\sum_{m=0}^3\binom{20}{m}=1351$ labelled networks.
Species exchange fixes the empty set and ten two-edge sets, giving
$(1351+11)/2=681$ representatives, with counts $1,10,100,570$ by size.
In F, exchange is followed by division of the vector field by $\rho$,
which restores the normalization with ratio $1/\rho$ and preserves stability.
Catalogue identifiers order the complexes as above, then order directed
edge sets lexicographically within each size.

\begin{theorem}\label{thm:main}
For the universe above, $\mathcal B_E=\mathcal B_F$ consists of the 27
representatives in Table~\ref{tab:mechanisms}. Every listed network admits
the rational E parameters in Table~\ref{tab:parameters}. The other 654
representatives are excluded for all positive F parameters.
\end{theorem}

\begin{theorem}\label{thm:minimal}
Three internal directed reactions are necessary and sufficient. All 27
positive representatives are deletion-minimal in both models. Within this
universe, bistability is equivalent to containing a listed mechanism up to
species exchange. Adding reactions at sufficiently small positive rates
preserves each positive witness under its removal constraint.
\end{theorem}

Each positive representative has one unary reaction and two quadratic-source
reactions: seventeen have $X\to X+Y$ and ten have $X\to2Y$ in canonical
order.
\input{paper/generated/mechanisms.tex}

The negative cases divide among four exclusion rules:
\begin{center}\input{paper/generated/coverage.tex}\end{center}

Fully open one- and two-non-flow-reaction networks have established
multistationarity classifications \cite{joshi1,joshishiu}. Our target
requires two sinks. Four-total-reaction classifications and small-switch
motifs \cite{bbh,wilhelm,panteatalk} overlap the constructions here, while
our positive cases have seven total reactions. The proof uses symbolic
stability and elimination ideas \cite{torresfeliu}, alongside the recent
equilibrium and cusp frameworks \cite{banajifeliu,banajidual,lixiviant}.
The supplied N0549 mechanism also occurs in template 02 of
\cite{lixiviant}, up to species exchange. SWITCHES \cite{switches} supplies
sampled examples in a different reaction alphabet. The result here decides
the entire fixed family under both removal constraints.

\section{Exclusion proofs}
Every allowed field has the form
\begin{align}
 f&=a+A x+B y-Cx^2+Dxy+Ey^2,\label{eq:f}\\
 g&=b+L x+G y+Hx^2+Ixy-Qy^2,\label{eq:g}
\end{align}
where $B,C,E,L,H,Q\ge0$. These signs follow from nonnegative production
of an absent species and the product molecularity bound.

\begin{lemma}\label{lem:geometry}
At every positive equilibrium, $f_x,g_y<0$. Each coordinate projection is
injective on the positive equilibrium set. Two hyperbolic sinks require
at least three positive equilibria and imply a positive saddle at an
arbitrarily nearby positive parameter vector.
\end{lemma}
\begin{proof}
At an equilibrium,
\[
 f_x=-\frac{a+By+Ey^2}{x}-Cx<0,\qquad
 g_y=-\frac{b+Lx+Hx^2}{y}-Qy<0.
\]
For fixed $x>0$, the constant term of $g$ in $y$ is positive. If $Q>0$,
there is exactly one positive root. If $Q=0$, there is one precisely
where $G+Ix<0$, an interval. Hence the positive $g$-nullcline is a smooth
graph $y=\phi(x)$ over an interval. Exchanging species proves the other
projection assertion. Set $h(x)=f(x,\phi(x))$. At a zero,
$h'=\det DF/g_y$, so a sink is a simple decreasing zero. Between two
such zeros, $h$ changes from negative to positive and has another zero.
Replacing $a$ by $a+\delta$ replaces $h$ by $h+\delta$ and preserves
both sinks for small $\delta$. Choose a regular value $-\delta$ on a
compact interval between them. The intermediate upward crossing is then
simple and has $\det DF<0$. Such values occur arbitrarily close to zero
by the one-dimensional Sard theorem.
\end{proof}

Consequently, a nonnegative determinant at every positive equilibrium
throughout a parameter family excludes bistability, including parameter
choices with degenerate equilibria.

\paragraph{JAC: 245 representatives.}
The expanded F-model determinant has nonnegative coefficients in
$x,y,\rho,k_1,\ldots,k_m$ and contains $\rho$ with coefficient one.
It is strictly positive. A quadratic field satisfies
\[
 F(p)-F(q)=DF((p+q)/2)(p-q),
\]
so two distinct positive zeros would force a singular Jacobian at their
midpoint.

\paragraph{DIV: 216 further representatives.}
Define
\[
 K=\begin{pmatrix}xf_x-f&yf_y\\xg_x&yg_y-g\end{pmatrix}.
\]
The polynomial $\det K$ has nonnegative coefficients and contains $ab$
with coefficient one. At an equilibrium $K=DF\diag(x,y)$, hence
$\det DF>0$. Lemma~\ref{lem:geometry} excludes two sinks.

\paragraph{ROOT: 179 further representatives.}
Eliminate one variable using the nominal-degree Sylvester resultant $P(t)$.
Its coefficients are polynomials in positive parameters. A coefficient
has fixed strict sign when all its nonzero parameter coefficients have
that sign; other signs range independently over $\{-1,0,1\}$.
Each certificate verifies at most two sign variations and at least one
fixed strict coefficient. Thus every positive parameter specialization
gives a nonzero polynomial with at most two positive roots by Descartes'
rule. Projection injectivity and Lemma~\ref{lem:geometry} give the exclusion.
This upper bound includes spurious roots, vanished coefficients and
degree drops. The strict coefficient rules out identically zero
specializations of the fixed Sylvester determinant.

For example, N0056 consists of $X\xrightarrow{k}X+Y$ and
$X+Y\xrightarrow{h}2X$. Its equilibrium equations reduce to
\[
 y=\frac{b+kx}{1+hx},\qquad
 a+(ah-\rho+hb)x+h(k-\rho)x^2=0.
\]
The constant coefficient $a>0$ makes the quadratic bound valid also at
$k=\rho$. This known multistationary motif therefore has insufficient
positive equilibria for two sinks.

\paragraph{FLUX: the last 14 negative representatives.}
These are N0350, N0354, N0355, N0361, N0369, N0391, N0600, N0610,
N0638, N0646, N0648, N0649, N0661 and N0673.
Let $S$ and $A$ be the internal stoichiometric and source matrices,
with species as rows, and put $v_j=k_jx^{\alpha_{j1}}y^{\alpha_{j2}}$.
At an equilibrium,
\[
 U=(a,b)^T+Sv=(\rho x,y)^T>0,\qquad
 DF\diag(x,y)=S\diag(v)A^T-\diag(U)=:M(w),
\]
where $w=(a,b,v)^T$. Write $\det M(w)=w^THw$.
All admissible equilibria lie in the cone
\[
 \mathcal K=\{w\ge0:U_X\ge0,\ U_Y\ge0\}.
\]
It is pointed and full dimensional; take $v_j=1$ and sufficiently large
$a,b$ for an interior point. The section $\sum_iw_i=1$ is a compact
polytope, so its extreme rays generate $\mathcal K$. For $m$ internal
edges, each ray is obtained from $m+1$ independent active rows among the
$m+4$ inequalities. Exhausting these subsets gives the primitive integer
ray matrix $R_*$. The certificates verify
\begin{equation}\label{eq:sos}
 R_*^THR_*=N+\sum_\ell c_\ell q_\ell q_\ell^T,
 \qquad N_{ij}\ge0,\quad c_\ell>0.
\end{equation}
Hence $w=R_*t$, $t\ge0$, gives $\det M(w)\ge0$ and the required
family-wide saddle exclusion. Eleven matrices are entrywise nonnegative.
For the remaining three, discard coordinates with entrywise nonnegative
rows. The reduced forms are
\begin{align*}
 \mathrm{N0391}:\quad u^2-2uv+4v^2&=(u-v)^2+3v^2,\\
 \mathrm{N0646}:\quad u^2-uv+v^2&=(u-v/2)^2+3v^2/4,\\
 \mathrm{N0610}:\quad u^2+3uv+2v^2-uw-vw+w^2
  &=\bigl(w-(u+v)/2\bigr)^2\\[-3pt]
  &\quad+3u^2/4+5uv/2+7v^2/4.
\end{align*}
The records give every ray and the full identity \eqref{eq:sos}.

\section{Existence certificates and minimality}
For each listed network, a certificate supplies positive rational
parameters with $\rho=1$, a square-free $P\in\mathbb Q[t]$, and
$x=t$, $y=N(t)/D(t)$. Both substituted equilibrium residual numerators
are divisible by $P$. Three disjoint rational intervals isolate positive
roots; singleton intervals specify simple rational roots. Sturm sequences
verify isolation. Rational interval evaluation verifies $D\ne0$, $y>0$,
negative trace and determinant signs $+,-,+$. Thus the same parameter
vector has two positive hyperbolic sinks.

\begin{proof}[Proof of Theorem~\ref{thm:main}]
The four rules exclude $245+216+179+14=654$ representatives in F.
The other 27 have exact E witnesses and therefore F witnesses.
These disjoint sets exhaust the independently enumerated catalogue.
\end{proof}

\begin{proof}[Proof of Theorem~\ref{thm:minimal}]
All networks with at most two internal edges are negative; every positive
has three. Each deletion is therefore negative. Assign each added edge
rate $\varepsilon c_j$ with $c_j>0$. At $\varepsilon=0$, both sink
Jacobians are nonsingular. The implicit function theorem continues the
distinct positive sinks for small $\varepsilon>0$, preserving their strict
Hurwitz inequalities and the removal constraint. This is the usual
reaction-addition inheritance operation \cite{banajipantea}.
\end{proof}

The flux matrix $M(w)$ is independent of the removal ratio. In fact a
single F equilibrium can be reconstructed in E at $z'=U$, with
$k'_j=v_j/(z')^{\alpha_j}$, preserving $DF\diag(z)$ and its determinant
sign. The rates depend on that equilibrium; the two-sink comparison is
established by the simultaneous witnesses above.

\section{An equal-removal switch}
The supplied baseline N0549, before canonical species exchange, is
\[
 Y\xrightarrow{10}2X,\qquad 2X\xrightarrow{50/9}X+Y,\qquad
 X+Y\xrightarrow{10}Y,
\]
with $a=4/99$, $b=1/90$ and removals $(1,1)$. Its equations give
\[
 y=\frac{1/90+(50/9)x^2}{11},\qquad
 11f=-\frac{500}{9}(x-1/10)(x-1/5)(x-3/5).
\]
The three equilibria have $x=1/10,1/5,3/5$ and respectively
$\det DF=25/9,-20/9,100/9$. Their trace is
$-12-(100/9)x-10y<0$, proving sink--saddle--sink order.

Varying $a>0$ while fixing the other rates yields
\[
 P_a(x)=-\frac{500}{9}x^3+50x^2-\frac{100}{9}x+11a+\frac29.
\]
The critical points are $x_\pm=(9\pm\sqrt{21})/30$. The positive fold
occurs at $x_-$ and
\[
 a_+=\frac1{99}+\frac{7\sqrt{21}}{891}\simeq0.0461033;
\]
the other fold value is negative. For $0<a<a_+$, $P_a(0)>0$ and the
two extrema have opposite signs, giving three positive simple roots with
$\det DF=-P'_a(x)$. For $a>a_+$, one positive
sink remains. At the fold, $P''_a(x_-)\ne0$ and $\partial_aP_a=11$,
verifying a nondegenerate saddle-node (Figure~\ref{fig:switch}).
\begin{figure}[htbp]\centering
\includegraphics[width=\textwidth]{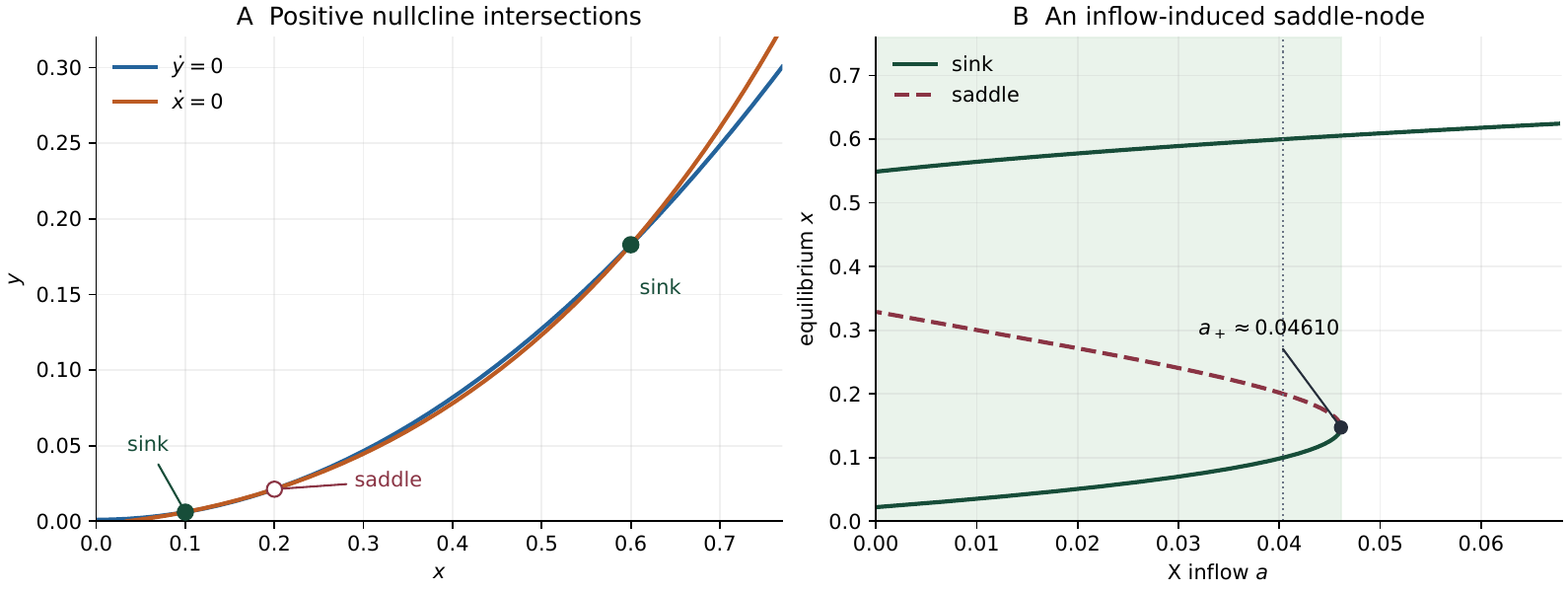}
\caption{N0549 in the displayed species order. Left: nullclines at
$a=4/99$ with two positive sinks and a saddle. Right: equilibrium branches
as $a$ varies; solid branches are sinks and the dashed branch is a saddle.
The shaded interval $0<a<a_+$ is bistable.}\label{fig:switch}
\end{figure}

\section{Computation and data}
The supplementary files contain all 1362 labels, 27 positive certificates
and 654 exclusion certificates. The checker reconstructs fields from
directed reactions using Python's standard-library rational arithmetic.
It verifies the sparse determinant identities, fixed Sylvester
determinants, complete flux-ray lists, square decompositions, Sturm
intervals and Hurwitz signs. Run \texttt{python verify\_all.py} to check
the classification, deletion minimality and the baseline fold identities.
The full rebuild command is \texttt{.\textbackslash reproduce.ps1 -Full};
dependencies, certificate formats and output paths are described in the
accompanying README. The catalogue and exact witness parameters are in
\path{data/catalogue.json} and \path{data/positive_witnesses.csv}.

\paragraph{Computational assistance.}
OpenAI Codex assisted with symbolic derivations, verification code and
manuscript preparation.

\appendix
\section{Rational parameters}
\begingroup\small
\input{paper/generated/parameters.tex}
\endgroup

\end{document}

%% file: paper/generated/mechanisms.tex
\begin{longtable}{@{}l l l l@{}}
\caption{The 27 deletion-minimal mechanisms, up to species exchange. Four positive flow reactions are implicit. Rates follow the displayed edge order.}\label{tab:mechanisms}\\
\toprule ID & Reaction 1 & Reaction 2 & Reaction 3 \\ \midrule\endfirsthead
\toprule ID & Reaction 1 & Reaction 2 & Reaction 3 \\ \midrule\endhead
\bottomrule\endfoot
N0438 & $X\to X{+}Y$ & $2X\to X$ & $X{+}Y\to 2X$ \\
N0440 & $X\to X{+}Y$ & $2X\to X$ & $2Y\to X$ \\
N0442 & $X\to X{+}Y$ & $2X\to X$ & $2Y\to 2X$ \\
N0443 & $X\to X{+}Y$ & $2X\to X$ & $2Y\to X{+}Y$ \\
N0448 & $X\to X{+}Y$ & $2X\to Y$ & $X{+}Y\to 2X$ \\
N0450 & $X\to X{+}Y$ & $2X\to Y$ & $2Y\to X$ \\
N0452 & $X\to X{+}Y$ & $2X\to Y$ & $2Y\to 2X$ \\
N0453 & $X\to X{+}Y$ & $2X\to Y$ & $2Y\to X{+}Y$ \\
N0459 & $X\to X{+}Y$ & $2X\to X{+}Y$ & $2Y\to X$ \\
N0467 & $X\to X{+}Y$ & $2X\to 2Y$ & $2Y\to X$ \\
N0474 & $X\to X{+}Y$ & $X{+}Y\to X$ & $2Y\to X$ \\
N0476 & $X\to X{+}Y$ & $X{+}Y\to X$ & $2Y\to 2X$ \\
N0477 & $X\to X{+}Y$ & $X{+}Y\to X$ & $2Y\to X{+}Y$ \\
N0480 & $X\to X{+}Y$ & $X{+}Y\to Y$ & $2Y\to X$ \\
N0482 & $X\to X{+}Y$ & $X{+}Y\to Y$ & $2Y\to 2X$ \\
N0483 & $X\to X{+}Y$ & $X{+}Y\to Y$ & $2Y\to X{+}Y$ \\
N0489 & $X\to X{+}Y$ & $X{+}Y\to 2Y$ & $2Y\to X$ \\
N0510 & $X\to 2Y$ & $2X\to X$ & $X{+}Y\to 2X$ \\
N0514 & $X\to 2Y$ & $2X\to X$ & $2Y\to 2X$ \\
N0515 & $X\to 2Y$ & $2X\to X$ & $2Y\to X{+}Y$ \\
N0520 & $X\to 2Y$ & $2X\to Y$ & $X{+}Y\to 2X$ \\
N0524 & $X\to 2Y$ & $2X\to Y$ & $2Y\to 2X$ \\
N0525 & $X\to 2Y$ & $2X\to Y$ & $2Y\to X{+}Y$ \\
N0548 & $X\to 2Y$ & $X{+}Y\to X$ & $2Y\to 2X$ \\
N0549 & $X\to 2Y$ & $X{+}Y\to X$ & $2Y\to X{+}Y$ \\
N0554 & $X\to 2Y$ & $X{+}Y\to Y$ & $2Y\to 2X$ \\
N0555 & $X\to 2Y$ & $X{+}Y\to Y$ & $2Y\to X{+}Y$ \\
\end{longtable}

%% file: paper/generated/coverage.tex
\begin{tabular}{@{}rrrrrrr@{}}
\toprule Internal edges & JAC & DIV & ROOT & FLUX & YES & Total \\ \midrule
0 & 1 & 0 & 0 & 0 & 0 & 1 \\
1 & 8 & 2 & 0 & 0 & 0 & 10 \\
2 & 54 & 33 & 13 & 0 & 0 & 100 \\
3 & 182 & 181 & 166 & 14 & 27 & 570 \\
\midrule Total & 245 & 216 & 179 & 14 & 27 & 681 \\
\bottomrule\end{tabular}

%% file: paper/generated/parameters.tex
\begin{longtable}{@{}lccccc@{}}
\caption{Rational witnesses for Table~\ref{tab:mechanisms}, with removals $(1,1)$ and rates in edge order. JSON certificates give exact equilibria and stability bounds.}\label{tab:parameters}\\
\toprule ID & $a$ & $b$ & $k_1$ & $k_2$ & $k_3$ \\ \midrule\endfirsthead
\toprule ID & $a$ & $b$ & $k_1$ & $k_2$ & $k_3$ \\ \midrule\endhead
\bottomrule\endfoot
N0438 & $\frac{1}{50}$ & $\frac{1}{50}$ & $\frac{401}{202}$ & $\frac{25}{10201}$ & $\frac{5025}{10201}$ \\[4pt]
N0440 & $\frac{1}{50}$ & $\frac{3}{100}$ & $\frac{601}{202}$ & $\frac{25}{10201}$ & $\frac{5025}{10201}$ \\[4pt]
N0442 & $\frac{1}{50}$ & $\frac{3}{100}$ & $\frac{43}{29}$ & $\frac{100}{41209}$ & $\frac{2525}{2601}$ \\[4pt]
N0443 & $\frac{1}{50}$ & $\frac{1}{50}$ & $\frac{401}{202}$ & $\frac{25}{10201}$ & $\frac{5025}{10201}$ \\[4pt]
N0448 & $\frac{3}{100}$ & $\frac{1}{50}$ & $\frac{401}{202}$ & $\frac{25}{10201}$ & $\frac{10050}{20503}$ \\[4pt]
N0450 & $\frac{3}{100}$ & $\frac{3}{100}$ & $\frac{601}{202}$ & $\frac{25}{10201}$ & $\frac{20100}{41209}$ \\[4pt]
N0452 & $\frac{3}{100}$ & $\frac{3}{100}$ & $\frac{43}{29}$ & $\frac{100}{41209}$ & $\frac{10100}{10609}$ \\[4pt]
N0453 & $\frac{3}{100}$ & $\frac{1}{50}$ & $\frac{401}{202}$ & $\frac{25}{10201}$ & $\frac{20100}{41209}$ \\[4pt]
N0459 & $\frac{1}{50}$ & $\frac{3}{100}$ & $\frac{601}{202}$ & $\frac{25}{10201}$ & $\frac{20100}{41209}$ \\[4pt]
N0467 & $\frac{3}{100}$ & $\frac{3}{100}$ & $\frac{601}{202}$ & $\frac{25}{10201}$ & $\frac{1675}{3468}$ \\[4pt]
N0474 & $\frac{1}{100}$ & $\frac{1}{25}$ & $\frac{601}{202}$ & $\frac{25}{10201}$ & $\frac{5025}{10201}$ \\[4pt]
N0476 & $\frac{1}{100}$ & $\frac{1}{25}$ & $\frac{43}{29}$ & $\frac{50}{10353}$ & $\frac{2525}{2601}$ \\[4pt]
N0477 & $\frac{1}{100}$ & $\frac{3}{100}$ & $\frac{401}{202}$ & $\frac{25}{10201}$ & $\frac{5025}{10201}$ \\[4pt]
N0480 & $\frac{1}{50}$ & $\frac{3}{100}$ & $\frac{601}{202}$ & $\frac{25}{10201}$ & $\frac{5025}{10201}$ \\[4pt]
N0482 & $\frac{1}{50}$ & $\frac{3}{100}$ & $\frac{43}{29}$ & $\frac{50}{10353}$ & $\frac{2525}{2601}$ \\[4pt]
N0483 & $\frac{1}{50}$ & $\frac{1}{50}$ & $\frac{401}{202}$ & $\frac{25}{10201}$ & $\frac{5025}{10201}$ \\[4pt]
N0489 & $\frac{1}{50}$ & $\frac{3}{100}$ & $\frac{601}{202}$ & $\frac{50}{20503}$ & $\frac{20100}{41209}$ \\[4pt]
N0510 & $\frac{3}{100}$ & $\frac{1}{50}$ & $\frac{101}{2}$ & $25$ & $\frac{5050}{103}$ \\[4pt]
N0514 & $\frac{3}{100}$ & $\frac{3}{100}$ & $17$ & $\frac{100}{9}$ & $\frac{2600}{2809}$ \\[4pt]
N0515 & $\frac{3}{100}$ & $\frac{1}{50}$ & $\frac{101}{2}$ & $25$ & $\frac{10100}{10609}$ \\[4pt]
N0520 & $\frac{1}{25}$ & $\frac{1}{50}$ & $\frac{101}{2}$ & $25$ & $\frac{2525}{52}$ \\[4pt]
N0524 & $\frac{1}{25}$ & $\frac{3}{100}$ & $17$ & $\frac{100}{9}$ & $\frac{650}{729}$ \\[4pt]
N0525 & $\frac{1}{25}$ & $\frac{1}{50}$ & $\frac{101}{2}$ & $25$ & $\frac{2525}{2704}$ \\[4pt]
N0548 & $\frac{1}{50}$ & $\frac{1}{25}$ & $17$ & $\frac{100}{159}$ & $\frac{2600}{2809}$ \\[4pt]
N0549 & $\frac{1}{50}$ & $\frac{3}{100}$ & $\frac{101}{2}$ & $\frac{50}{103}$ & $\frac{10100}{10609}$ \\[4pt]
N0554 & $\frac{3}{100}$ & $\frac{3}{100}$ & $17$ & $\frac{100}{159}$ & $\frac{2600}{2809}$ \\[4pt]
N0555 & $\frac{3}{100}$ & $\frac{1}{50}$ & $\frac{101}{2}$ & $\frac{50}{103}$ & $\frac{10100}{10609}$ \\[4pt]
\end{longtable}